\documentclass[final,leqno]{siamart250211}
\usepackage{graphicx} 

\newcommand{\N}{\mathbb N}

\newcommand{\R}{\mathbb R}

\newcommand{\pa}{\partial}

\newcommand{\la}{\label}
\newcommand{\fr}{\frac}
\newcommand{\na}{\nabla}
\newcommand{\be}{\begin{equation}}
\newcommand{\ee}{\end{equation}}
\newcommand{\ba}{\begin{array}{l}}
\newcommand{\ea}{\end{array}}

\newcommand{\beg}{\begin}

\newcommand{\D}{\Delta}
\renewcommand{\l}{\Lambda_D}

\usepackage{cite}
\usepackage{mathtools}

\usepackage{xcolor}
\usepackage{graphicx}
\usepackage[]{amsfonts}
\usepackage{amsmath}
\usepackage {amssymb}
\usepackage{bm}
\usepackage{mathrsfs}
\usepackage{setspace}
\usepackage{amsmath}
 
\numberwithin{lem}{section}
\newtheorem{prop}{Proposition}
\numberwithin{prop}{section}
\newtheorem{Thm}{Theorem}[section]
\newtheorem{rem}{Remark}[section]
\numberwithin{rem}{section}

\usepackage{MnSymbol,wasysym}

\def\RR{{\mathbb R}}

\def\NN{\mathbb N}

\def\RR{{\mathbb R}}

\def\NN{\mathbb N}

\title{Convergence Analysis of PINNs for Fractional Diffusion Equations in Bounded Domains}

\author{
Elie Abdo\thanks{Department of Mathematics, American University of Beirut, Beirut, 1107 2020, Lebanon. (\texttt{ea94@aub.edu.lb}).}
\and
Lihui Chai\thanks{School of Mathematics, Sun Yat-sen University, Guangzhou, 510275, China. (\texttt{chailihui@mail.sysu.edu.cn}).}
\and
Ruimeng Hu\thanks{Department of Mathematics, Department of Statistics and Applied Probability, University of California, Santa Barbara, CA 93106-3080, USA. (\texttt{rhu@ucsb.edu}).}
\and
Xu Yang\thanks{Department of Mathematics, University of California, Santa Barbara, CA 93106-3080, USA. (\texttt{xuyang@math.ucsb.edu}).}
}
\date{}

\begin{document}

\maketitle

\begin{abstract} 
We establish the convergence of physics-informed neural networks (PINNs) for time-dependent fractional diffusion equations posed on bounded domains. The presence of fractional Laplacian operators introduces nonlocal behavior and regularity constraints, and standard neural network approximations do not naturally enforce the associated spectral boundary conditions. To address this challenge, we introduce a spectrally-defined mollification strategy that preserves the structure of the nonlocal operator while ensuring boundary compatibility. This enables the derivation of rigorous energy estimates in Sobolev spaces. Our results rely on analytical tools from PDE theory, highlighting the compatibility of PINN approximations with classical energy estimates for nonlocal equations. We prove convergence of the PINN approximation in any space-time Sobolev norm $H^k$ (with $k \in \N)$. The analysis highlights the role of mollified residuals in enabling theoretical guarantees for neural-network-based solvers of nonlocal PDEs.
\end{abstract}

\section{Introduction}

Fractional partial differential equations (PDEs) have emerged as effective models for describing anomalous transport, long range interactions, and memory effects in various physical and engineering systems. Applications include subsurface flow in geophysics~\cite{benson2000application}, edge detection in image processing~\cite{gilboa2009nonlocal}, turbulent diffusion~\cite{metzler2000random}, and option pricing in mathematical finance~\cite{cartea2007fractional}. Among these models, the fractional diffusion equation, in which the classical Laplacian is replaced by a fractional Laplace operator, captures the nonlocal nature of transport processes governed by heavy tailed jump distributions or Lévy flights. This nonlocality gives rise to qualitatively different dynamics compared to standard diffusion and introduces substantial challenges in both analysis and computation.

On bounded domains, the use of fractional Laplacians introduces both analytical and computational complexities. Chief among these are the nonlocal nature of the operator and the proper treatment of boundary conditions. Unlike their local counterparts, solutions to fractional diffusion equations are influenced by the behavior of the function over the entire domain, and enforcing Dirichlet conditions requires nonstandard formulations. Various interpretations of the fractional Laplacian exist in bounded domains, including spectral, restricted, and regional definitions, each with distinct analytical and numerical implications~\cite{caffarelli2016fractional, constantin2017some}. In this work, we focus on the spectral definition of the Dirichlet fractional Laplacian, which arises naturally from the eigenstructure of the classical Laplacian and is widely used in physical models. For recent studies on the Dirichlet fractional Laplacian and its role in nonlocal active scalar equations on bounded domains, we refer to~\cite{abdo2024dirichlet}.

Physics-informed neural networks (PINNs) have been proposed as a flexible mesh free approach for solving PDEs by combining the expressive power of neural networks with physics based loss functions~\cite{jin2023asymptotic, jin2024asymptotic, wan2025error, raissi2019physics, E2018, li2020fourier, abdo2025neural, lu2022neural, zhu2025physicssolver, chen2025structure}. While PINNs have shown promising performance in a variety of settings~\cite{deepGalerkin2018, zang2020weak, cai2021least}, rigorous theoretical guarantees remain limited, especially in nonlocal and fractional contexts. Recent works~\cite{lu2021learning, beck2020overview, hu2023higher, abdo2025neural} have analyzed PINN approximation errors for local PDEs under strong smoothness assumptions and idealized optimization settings. However, comparatively fewer results address convergence of PINNs when the underlying operator is nonlocal and defined on bounded domains with nontrivial boundary conditions. In particular, the spectral Dirichlet fractional Laplacian presents two main difficulties. First, standard network architectures do not naturally satisfy spectral boundary conditions. Second, the nonlocal operator acts globally, which complicates the use of classical energy methods. Similar issues have been observed in periodic settings, where incorporating boundary conditions into the network architecture can significantly improve both stability and accuracy~\cite{hao2024structure}.

In a recent work~\cite{abdo2024error}, a convergence theory was developed for PINNs approximating kinetic equations, specifically the Boltzmann equation near equilibrium. While that setting involved nonlocality in velocity space, the present work addresses spatial nonlocality on bounded domains, leading to fundamentally different analytical challenges. Our goal is to extend rigorous convergence analysis to fractional diffusion operators in bounded domains, where standard numerical methods often suffer from increased complexity due to nonlocal interactions and irregular boundaries.

To address these challenges, we introduce a spectrally defined mollification strategy that modifies the residual loss to ensure compatibility with the Dirichlet boundary conditions. This approach enables a rigorous convergence analysis in Sobolev spaces. In particular, we establish error bounds for the PINN approximation in any space-time Sobolev norm $H^k$ (with $k \in \N)$.  Our results contribute to the mathematical foundations of PINNs and provide theoretical justification for their application to nonlocal diffusion equations posed on bounded domains. While our work is motivated by numerical methods for nonlocal PDEs, the main techniques developed in this paper are grounded in the tools of PDE analysis and functional estimates, rather than in discretization or algorithmic schemes. As such, our contribution is primarily analytical in nature.

We begin by recalling the definition and properties of the spectral Dirichlet fractional Laplacian on bounded domains.

\subsection{Dirichlet fractional Laplacian} Let $\Omega$ be a bounded smooth domain in $\R^d$. We denote by $\Delta_D$ the Laplacian operator with homogeneous Dirichlet boundary conditions. We note that $-\Delta_D$  is defined on $\mathcal{D}(-\D_D) = H^2(\Omega) \cap H_0^1(\Omega)$, and is positive and self-adjoint in $L^2(\Omega)$. Then there exists an orthonormal basis of $L^2(\Omega)$ consisting of eigenfunctions $\left\{w_j\right\}_{j=1}^{\infty} \subset H_0^1(\Omega)$ of $-\Delta_D$ satisfying
\begin{equation}
    -\Delta_D w_j = \lambda_j w_j,
\end{equation}
where the eigenvalues $\lambda_j$ obey $0 < \lambda_1 \leq ... \leq \lambda_j \le ...\rightarrow \infty$.
For $s \in \R$, we define the fractional Laplacian operator of order $s$, denoted by $\l^s$, as
\begin{equation} \la{maindef}
    \l^s h= \sum_{j=1}^{\infty} \lambda_j^{\fr{s}{2}} (h, w_j)_{L^2} w_j,
\end{equation} with domain
\begin{equation}
    \mathcal{D}(\l^s) = \left\{h  : \|\l^{s} h\|_{L^2}^2 := \sum\limits_{j=1}^\infty \lambda_j^{s}(h, w_j)_{L^2}^2 < \infty \right\}.
\end{equation}
In particular, when $s>0$ the space $\mathcal D(\l^{-s})$ is understood as the dual space of $\mathcal D(\l^{s})$. It is evident that  $\mathcal D(\l^{s_2}) \subset \mathcal D(\l^{s_1})$ provided that $s_1  \le s_2$.
For $s \in [0,1]$, we identify the domains $\mathcal{D}(\l^s)$ with the usual Sobolev spaces as follows,
\begin{equation} \label{identification}
    \mathcal{D}(\l^s) = \begin{cases} H^{s}(\Omega), \hspace{8cm} \mathrm{if} \; s \in [0, \fr{1}{2}),
        \\ H_{00}^{\fr{1}{2}} (\Omega) = \left\{h \in H_0^\fr{1}{2} (\Omega) : h/\sqrt{d(x)} \in L^2(\Omega)\right\}, \hspace{2.5cm} \mathrm{if} \; s = \fr{1}{2},
        \\H_{0}^{s} (\Omega), \hspace{8cm} \mathrm{if} \; s \in (\fr{1}{2},1],
    \end{cases}
\end{equation} where $H_0^{s}(\Omega)$ is the Hilbert subspace of $H^s(\Omega)$ with vanishing boundary trace elements, and $d(x)$ is the distance to the boundary function.

We recall the identity
\begin{equation}
    \lambda^{\fr{s}{2}} = c_s \int_{0}^{\infty} t^{-1-\fr{s}{2}} (1-e^{-t\lambda}) dt,
\end{equation}
that holds for $s \in (0,2)$, where $c_s$ is given by
\begin{equation}
    1 = c_s \int_{0}^{\infty} t^{-1-\fr{s}{2}} (1-e^{-t}) dt.
\end{equation} Using the latter, we  obtain the integral representation
\begin{equation} \la{intrep}
    (\l^s f)(x) = c_s \int_{0}^{\infty} [f(x) - e^{t\Delta_D}f(x)]t^{-1 - \fr{s}{2}} dt,
\end{equation} for $f \in \mathcal{D}(\l^s)$ and $s \in (0,2)$. Here the heat operator $e^{t\Delta_D}$ is defined as
\begin{equation} \label{eqn:etdelta}
    (e^{t\Delta_D}f)(x) = \int_{\Omega} H_D(x, y, t) f(y) dy, 
\end{equation} with kernel $H_D(x,y, t)$ given by
\begin{equation} \label{eqn:heat-kernel}
    H_D(x,y,t) = \sum_{j=1}^{\infty} e^{-t\lambda_j} w_j(x) w_j(y).
\end{equation}

For $\epsilon \in (0,1)$, we let $J_{\epsilon}$ be the spectrally regularizing operator defined in terms of the heat semigroup $e^{t\Delta_D}$ by
\be
J_{\epsilon} \theta(x)
= \frac{-1}{\ln \epsilon} \int_{\epsilon}^{\fr{1}{\epsilon}} \frac{e^{t\Delta_D}\theta(x)}{t} dt.
\ee This family of regularizers not only smooths out an $L^2$ function but also ensures that all its higher-order Laplacians vanish on the boundary of $\Omega$. That is, given a function $\theta \in L^2(\Omega)$, we have $J_{\epsilon}\theta \in \mathcal{D}(\l^k)$ for any $k \in \mathbb{N}$ (see Proposition~\ref{thm21}).

\subsection{Residuals and errors} Let $\alpha \in [0,2]$. We are interested in approximating solutions to advection-diffusion equations
\be \la{fracheat}
\pa_t \psi + u \cdot \na \psi + \Lambda_D^{\alpha} \psi = f,
\ee with boundary conditions 
\be 
\psi|_{\pa \Omega} = 0,
\ee and initial data 
\be 
\psi(x,0) = \psi_0(x),
\ee 
by PINNs. Here $u:= u(x,t)$ is a given smooth divergence-free vector field obeying 
\be 
u \cdot n|_{\pa \Omega} = 0,
\ee where $n$ is the outward unit normal to $\pa \Omega$, and $f:=f(x,t)$ is a given smooth function such that $(-\Delta)^k f|_{\pa \Omega} = 0$ for any $k \in \N$

Neural networks do not necessarily vanish on the boundary of $\Omega$, and therefore they do not necessarily belong to $\mathcal{D}(\Lambda_D^{\alpha})$ when $\alpha \ge 1/2$. In order to overcome this challenge, we make use of the mollifiers $J_{\epsilon}$ to define the following PDE residual
\be
\mathcal{R}_i [\theta](x,t) = \pa_t J_{\epsilon} \psi_{\theta} + u \cdot \na J_{\epsilon} \psi_{\theta} + \l^{\alpha} J_{\epsilon} \psi_{\theta} - f,
\ee the initial residual 
\be 
\mathcal{R}_{t}[k;\theta] = \sum\limits_{i=1}^{k} [(\pa_t^{(i)} \psi_{\theta})  - (\pa_t^{(i)} \psi)](x,0),
\ee and the boundary residual
\be 
\mathcal{R}_{b}[\theta](x,t) =  \psi_{\theta}|_{\pa \Omega}.
\ee These residuals are well-defined because $J_{\epsilon}\psi_{\theta} \in \mathcal{D}(\l^k)$ for all $k \in \N$.  We point out that the time derivative term $(\pa_t^{(i)} \psi)(x,0)$ appearing in the definition of the initial residual depends solely on the spatial derivatives of the given initial data $\psi_0$ as $\psi$ is the exact smooth solution of the PDE \eqref{fracheat}. For regularity indices $\ell, k \in \N$, we define the total error $\mathcal{E}[\ell, k; \theta]$ by 
\be 
\mathcal{E}[\ell, k;\theta]^2 = \int_{0}^{T} \sum\limits_{i=0}^{k} \|\l^{\ell} \pa_t^{(i)}(\psi  - J_{\epsilon}\psi_{\theta})\|_{L^2}^2dt.
\ee We note that $\mathcal{E}[\ell,k;\theta]$ is well-defined because the solution $\psi$ is in $\mathcal{D}(\l^j)$ for any $j \in \N$ (see Appendix \ref{exandun}) and $J_{\epsilon} \psi_{\theta} \in \mathcal{D}(\l^j)$ for any $j \in \N$. Moreover, $\mathcal{E}[k,k;\theta]$ dominates the Sobolev $H^k([0,T] \times \Omega)$ norm of $\psi - J_{\epsilon}\psi_{\theta}$ (see Proposition \ref{dom}) and consequently, $\mathcal{E}[\ell,k;\theta]$ measures the distance between the true solution and the neural network in all space-time Sobolev spaces.

For a time regularity index $k \in \N$ and a spatial regularity index $\ell \in \N$, we define the generalization error $\mathcal{E}_G[\ell, k;\theta]$ by 
\be 
\mathcal{E}_G[\ell, k;\theta]^2
= \mathcal{E}_G^i[\ell, k;\theta]^2 + \mathcal{E}_G^t[\ell, k;\theta]^2 + \mathcal{E}_G^b[\theta]^2, 
\ee where 
\be 
\mathcal{E}_G^i[\ell, k;\theta]^2
= \begin{cases}
\int_{0}^{T} \sum\limits_{j=0}^{k} \|(-\Delta)^{\frac{\ell}{2}}\pa_t^{(j)} \mathcal{R}_i[\theta]\|_{L^2}^2 dt, \hspace{1.5cm} \mathrm{if \; \ell \; is \; even}
\\ \int_{0}^{T} \sum\limits_{j=0}^{k} \|\na(-\Delta)^{\frac{\ell-1}{2}}\pa_t^{(j)} \mathcal{R}_i[\theta]\|_{L^2}^2 dt, \hspace{1cm} \mathrm{if \; \ell \; is \; odd}
\end{cases},
\ee
\be 
\mathcal{E}_G^t[\ell,k;\theta]^2
= \begin{cases}
\|(-\Delta)^{\frac{\ell}{2}} \mathcal{R}_t[k;\theta]\|_{L^2}^2 , \hspace{1.5cm} \mathrm{if \; \ell \; is \; even}
\\ \|\na(-\Delta)^{\frac{\ell-1}{2}} \mathcal{R}_t[k;\theta]\|_{L^2}^2 , \hspace{1cm} \mathrm{if \; \ell \; is \; odd}
\end{cases},
\ee and 
\be 
\mathcal{E}_G^b[\theta]^2
= 
\int_{0}^{T} \int_{\pa \Omega }\mathcal{R}_b[\theta]^2 d\sigma(x).
\ee
We note that these errors are also well-defined because the residuals are smooth functions and the involved operators are all local. 

\subsection{Organization of the paper} The remainder of the paper is organized as follows. In Section~\ref{sec:regularizer}, we introduce the mollification strategy and establish essential analytical properties of the spectrally-defined regularizers, which play a key role in enforcing boundary compatibility for the PINN residual. Section~\ref{sec:GEE} is devoted to the generalization error analysis, where we derive upper bounds for the PINN residuals under suitable smoothness assumptions on the target solution. In Section~\ref{sec:TEE}, we combine the approximation and generalization analysis to obtain total error estimates for the PINN solution in any space-time Sobolev norm $H^k$ (with $k \in \N)$.  Section~\ref{sec:conclusion} provides concluding remarks and discusses directions for future work. Finally, the existence of unique smooth solutions to fractional reaction-diffusion equations on bounded smooth domains is established in Appendix \ref{exandun}. 

\section{Properties of the Regularizers}\label{sec:regularizer}
In this section, we investigate key properties of the family of mollification operators $J_{\epsilon}$, which play a central role in the convergence analysis of PINNs developed in later sections.

 The operator $J_{\epsilon}$ is uniformly bounded in the norm of $\mathcal{D}(\l^s)$ and smoothes out any function in $L^2$:

\begin{prop} \la{thm21}  Let $s$ be a real number and $\epsilon \in (0,1)$ be a small positive number. There exists a positive number $C$ depending only on $s$ such that
    \be \label{prop32}
    \|\l^s J_{\epsilon} \theta\|_{L^2} \le C\|\l^s\theta\|_{L^2},
    \ee
    provided that $\theta \in \mathcal{D}(\l^{s})$. Moreover,  for any real number $s\geq 0$ , it  holds that 
\begin{equation}
    \|\Lambda_D^{s} J_{\epsilon} \theta\|_{L^2} \le C \epsilon^{-\frac{s}2} \|\theta\|_{L^2},
\end{equation} for $\theta\in L^2(\Omega)$.
\end{prop}

The proof of Proposition~\ref{thm21} follows closely from that of Lemma 1 in~\cite{abdo2024regularity} and is therefore omitted here for brevity.

We next state a key commutation property between the fractional Laplacian $\l^s$ and the regularization operators $J_{\epsilon}$, which plays a central role in our analysis. This result ensures that the mollification procedure is compatible with the nonlocal structure of the problem. Specifically, the operators $\l^s$ and $J_{\epsilon}$ commute on $\mathcal{D}(\l^s)$:

\begin{prop}\label{prop:commute}
Let $\epsilon \in (0,1)$, $s > 0$, and $f \in \mathcal{D}(\l^s)$. Then
\[
\l^s J_{\epsilon} f(x) = J_{\epsilon} \l^s f(x),
\]
for almost every $x \in \Omega$.
\end{prop}

We refer the reader to~\cite{abdo2024dirichlet} for a detailed proof of Proposition~\ref{prop:commute}.

We also observe that the integer powers of the Laplacian commute with the operators $J_{\epsilon}$ without any boundary assumptions. In particular, this commutation property holds for $H^{2k}$ functions, with $k 
\in \N$, which do not necessarily satisfy homogeneous boundary conditions. We state and prove this result below.

\beg{prop} \label{lapcom} Let $\epsilon \in (0,1)$. Let $f \in H^2(\Omega)$. Then 
\be 
\Delta J_{\epsilon} f(x) = J_{\epsilon} \Delta f(x),
\ee for a.e. $x \in \Omega$. Consequently, if $k \in \NN$ and $f \in H^{2k}(\Omega)$, then 
\be 
(-\Delta)^k J_{\epsilon}f (x) = J_{\epsilon}(-\Delta)^k f(x),
\ee for a.e. $x \in \Omega$.
\end{prop}

\begin{proof} The commutator $T_{\epsilon} = J_{\epsilon} \Delta - \Delta J_{\epsilon}$ vanishes on $C_{0}^{\infty}(\Omega)$ because $J_{\epsilon}$ and $-\Delta = \l^2$ commutes on $\mathcal{D}(\l^2)$ (see Proposition \ref{prop:commute}). In view of the density of $C_{0}^{\infty}(\Omega)$ in $L^2(\Omega)$, $T_{\epsilon}$ extends uniquely to a bounded linear operator $\tilde{T}_{\epsilon}$ on $L^2(\Omega)$ such that $\tilde{T}_{\epsilon}h =0$ for any $h \in L^2(\Omega)$. However, the operator $T_{\epsilon}$ is well-defined on $H^2(\Omega)$. Moreover, it holds that 
\be \la{14}
\|T_{\epsilon}f\|_{L^2} \le C_{\epsilon}\|f\|_{L^2}.
\ee In order to prove the latter, we have, on the one hand, that 
\be \label{12}
\|\Delta J_{\epsilon} f\|_{L^2} 
\le C_{\epsilon}\|f\|_{L^2},
\ee due to \eqref{prop32}. On the other hand, $J_{\epsilon} \Delta f$ has the following eigenfunction expansion,
\be 
J_{\epsilon} \Delta f = \sum\limits_{j=1}^{\infty} \left(-\frac{1}{\ln \epsilon} \int_{\epsilon}^{1/\epsilon} \frac{e^{-t\lambda_j}}{t} dt \right) (\Delta f, \omega_j)_{L^2} \omega_j,
\ee and thus, its $L^2$ norm obeys 
\be 
\|J_{\epsilon}\Delta f\|_{L^2}^2
= \sum\limits_{j=1}^{\infty} \left(-\frac{1}{\ln \epsilon}\int_{\epsilon}^{1/\epsilon} \frac{e^{-t\lambda_j}}{t} dt \right)^2 (\Delta f, \omega_j)_{L^2}^2.
\ee As $t \ge \epsilon$, it holds that $e^{-t\lambda_j} \le e^{-\epsilon \lambda_j}$, and consequently,
\be 
\|J_{\epsilon}\Delta f\|_{L^2}^2 
\le \sum\limits_{j=1}^{\infty} \left(-\frac{1}{\ln \epsilon} \int_{\epsilon}^{1/\epsilon} \frac{1}{t}dt\right)^2 e^{-\epsilon \lambda_j} (\Delta f, \omega_j)_{L^2}^2
= \sum\limits_{j=1}^{\infty} 4e^{-\epsilon \lambda_j} (\Delta f, \omega_j)_{L^2}^2.
\ee In view of the boundedness of the operator $(-\Delta_D)^{-1} \Delta$ from $L^2(\Omega)$ to $L^2(\Omega)$, we have 
\be 
\beg{aligned}
|(\Delta f, \omega_j)_{L^2}|
&= |(-\Delta_D)(-\Delta_D)^{-1} \Delta f, \omega_j)_{L^2}
= |(-\Delta_D)^{-1}\Delta f, -\Delta_D \omega_j)_{L^2}|
\\&\le \|(-\Delta_D)^{-1}\Delta f\|_{L^2} \|\Delta \omega_j\|_{L^2}
\le C\lambda_j \|f\|_{L^2}, 
\end{aligned}
\ee 
after integrating by parts and using the fact that $\omega_j$ and $(-\Delta_D)^{-1} \Delta f$ vanish on the boundary of $\Omega$.  Consequently,
\be \label{13}
\|J_{\epsilon}\Delta f\|_{L^2}^2
\le C\|f\|_{L^2}^2 \sum\limits_{j=1}^{\infty} \lambda_j^2 e^{-\epsilon \lambda_j}
\le C_{\epsilon}\|f\|_{L^2}^2.
\ee  Putting \eqref{12} and \eqref{13} together gives \eqref{14}. This shows that $T_{\epsilon}$ is a bounded linear operator from $(H^2(\Omega), \|\cdot\|_{L^2})$ to $(L^2(\Omega), \|\cdot\|_{L^2})$. By the uniqueness of the extension, we infer that 
\be 
\tilde{T}_{\epsilon} = J_{\epsilon}\Delta - \Delta J_{\epsilon}, 
\ee on $(H^2(\Omega), \|\cdot\|_{L^2})$. But $\tilde{T}_{\epsilon}$ vanishes on $L^2(\Omega)$,  thus 
\be 
J_{\epsilon}\Delta - \Delta J_{\epsilon} = 0, 
\ee on $H^2(\Omega)$. 
\end{proof} 

We derive new quantitative bounds for the convergence of the regularizers $J_{\epsilon}$ in the norms of $\mathcal{D}(\l^s)$.

\beg{prop} \label{quantitativeconv}
Let $\epsilon \in (0,1)$ and $s \ge 0$ be real numbers. Let $\psi \in \mathcal{D}(\l^{2s+d})$. There exists a real-valued positive function $\kappa(\epsilon)$ such that $\kappa(\epsilon) \rightarrow 0$  as $\epsilon \rightarrow 0$ and 
\be 
\|\l^s(J_{\epsilon}f - f)\|_{L^2} 
\le \kappa(\epsilon) \|f\|_{L^2}^{\frac{1}{2}}\|\l^{2s+d}f\|_{L^2}^{\frac{1}{2}}.
\ee 
\end{prop}

\begin{proof}
The eigenfunction expansion of $J_{\epsilon}f - f$ is given by 
\be 
J_{\epsilon}f - f = \sum\limits_{j=1}^{\infty} (J_{\epsilon}f - f, \omega_j)_{L^2} \omega_j,
\ee and consequently, 
\be 
\|\l^s (J_{\epsilon}f - f)\|_{L^2}^2 = \sum\limits_{j=1}^{\infty} \lambda_j^s (J_{\epsilon} f - f, \omega_j)_{L^2}^2.
\ee Using the integral representation formula of $J_{\epsilon}$ and the identity 
\be 
\frac{1}{\ln \epsilon}\int_{\epsilon}^{1/\epsilon} \frac{1}{2t} dt = -1,
\ee we have  
\be 
\begin{aligned}
(J_{\epsilon}f - f, \omega_j)_{L^2} 
&= \left[- \frac{1}{\ln \epsilon} \int_{\epsilon}^{1/\epsilon} \frac{e^{-t\lambda_j}}{t} dt -1 \right](f, \omega_j)_{L^2} 
\\&= \left[\frac{1}{\ln \epsilon}\int_{\epsilon}^{1/\epsilon} \frac{1-2e^{-t\lambda_j}}{2t} dt\right] (f, \omega_j)_{L^2},
\end{aligned}
\ee from which we obtain 
\be 
\|\l^s (J_{\epsilon}f - f)\|_{L^2}^2
= \sum\limits_{j=1}^{\infty} \left[\frac{1}{\ln \epsilon}\int_{\epsilon}^{1/\epsilon} \frac{1-2e^{-t\lambda_j}}{2t} dt \right]^2 \lambda_j^s (f, \omega_j)_{L^2}^2.
\ee An application of the Cauchy-Schwarz inequality yields
\be 
\|\l^s (J_{\epsilon}f - f)\|_{L^2}^2 
\le \left(\sum\limits_{j=1}^{\infty} \frac{1}{\lambda_j^{d}} \left[\frac{1}{\ln \epsilon} \int_{\epsilon}^{1/\epsilon} \frac{1-2e^{-t\lambda_j}}{2t} dt \right]^4 \right)^{\frac{1}{2}} \left(\sum\limits_{j=1}^{\infty} \lambda_j^{2s+d} (f, \omega_j)_{L^2}^4 \right)^{\frac{1}{2}}.
\ee We let 
\be 
\kappa(\epsilon) := \left(\sum\limits_{j=1}^{\infty} \frac{1}{\lambda_j^d} \left[\frac{1}{\ln \epsilon} \int_{\epsilon}^{1/\epsilon} \frac{1-2e^{-t\lambda_j}}{2t} dt \right]^4 \right)^{\frac{1}{4}}. 
\ee Since  
\be 
\left|\frac{1}{\ln \epsilon} \int_{\epsilon}^{1/\epsilon} \frac{1-2e^{-t\lambda_j}}{2t} dt\right|
\le \frac{1}{|\ln \epsilon|} \int_{\epsilon}^{1/\epsilon}\frac{1}{2t} dt \le 1,
\ee and $\lambda_j \ge cj^{\frac{2}{d}}$ for any $j \in \NN$, it follows that 
\be 
\kappa(\epsilon)^4 \le \sum\limits_{j=1}^{\infty} \frac{1}{\lambda_j^d} \le C\sum\limits_{j=1}^{\infty} \frac{1}{j^2} <  \infty,
\ee and so $\kappa(\epsilon)$ is well-defined. Moreover, since 
\be 
\lim\limits_{\epsilon \to 0} \frac{1}{\ln \epsilon} \int_{\epsilon}^{1/\epsilon} \frac{1-2e^{-t\lambda_j}}{2t} dt = 0, 
\ee it follows that $\kappa(\epsilon) \rightarrow 0$ as $\epsilon \rightarrow 0$ by the Lebesgue Dominated Convergence Theorem. Finally, since $\|\omega_j\|_{L^2} = 1$ for any $j \in \N$, we have 
\be 
\beg{aligned}
\sum\limits_{j=1}^{\infty} \lambda_j^{2s+d} (f, \omega_j)_{L^2}^4
&\le \sum\limits_{j=1}^{\infty} \lambda_j^{2s+d} (f, \omega_j)_{L^2}^2 \|f\|_{L^2}^2 \|\omega_j\|_{L^2}^2
\\&\le \|f\|_{L^2}^2 \sum\limits_{j=1}^{\infty} \lambda_j^{2s+d} (f, \omega_j)_{L^2}^2  \le \|f\|_{L^2}^2 \|\l^{2s+d} f\|_{L^2}^2.
\end{aligned}
\ee Therefore, we obtain the desired quantitative bound
\be 
\|\l^s(J_{\epsilon}f -f)\|_{L^2}^2 \le \kappa(\epsilon)^2 \|f\|_{L^2}\|\l^{2s+d}f\|_{L^2}.
\ee 
\end{proof}

\section{Generalization Error Estimates}\label{sec:GEE}
In this section, we establish bounds on the generalization error of physics-informed neural networks for approximating solutions to fractional diffusion equations on bounded domains. Our main result shows that, under suitable regularity assumptions, one can construct a neural network whose generalization error is arbitrarily small. A key ingredient in the analysis is the approximation capability of deep neural networks for smooth functions in Sobolev norms. The following proposition, adapted from classical universal approximation theorems, provides the quantitative rate at which neural networks can approximate smooth target functions in $W^{k,\infty}$ norms.

\begin{prop} \label{existenceneural}
Let $\Omega \subset \RR^d$ be a smooth bounded domain. Let $T > 0$ be arbitrary and $k \in \NN$. Let $f \in C^{k+1} ([0,T] \times \bar{\Omega})$. Then there exists a function $\phi_N$ represented by a deep neural network with complexity $N$ such that 
\be \label{universal}
\|f - \phi_N\|_{W^{k,\infty}([0,T] \times \bar{\Omega})} \le \frac{1}{N^{\frac{1}{d+1}}} \|f\|_{W^{k+1, \infty}([0,T] \times \bar{\Omega})}.
\ee
\end{prop}

\begin{proof}
Since $[0,T] \times \bar{\Omega}$ is a compact subset of $\R^{d+1}$ and $f \in C^{k}([0,T] \times \bar{\Omega})$, we can apply Theorem 2.1 in \cite{biswas2022error} and deduce the existence of a deep neural network $\phi_N$ with complexity $N$ such that 
\be 
\|D^{\alpha}f - D^{\alpha} \phi_N\|_{L^{\infty}([0,T] \times \bar{\Omega})} 
\le \frac{C}{N^{\frac{k-|\alpha|}{d+1}}} \sup\limits_{|(x,t) - (y,t)| \le N^{\frac{-1}{d+1}}} |D^{\beta} f(x,t) - D^{\beta} f(y,t)|,
\ee  for any $|\alpha| \le k$ and $|\beta| = k$. Since $N \ge 1$ and $k \ge |\alpha|$, it follows that 
\be 
\|D^{\alpha}f - D^{\alpha} \phi_N\|_{L^{\infty}([0,T] \times \bar{\Omega})} 
\le C\sup\limits_{|(x,t) - (y,t)| \le N^{\frac{-1}{d+1}}} |D^{\beta} f(x,t) - D^{\beta} f(y,t)|,
\ee for any $|\alpha| \le k$ and $|\beta| =k$. Thus, 
\be 
\beg{aligned}
\|D^{\alpha}f - D^{\alpha} \phi_N\|_{L^{\infty}([0,T] \times \bar{\Omega})} 
&\le C\sup\limits_{|(x,t) - (y,t)| \le N^{\frac{-1}{d+1}}} \left[\frac{|D^{\beta} f(x,t) - D^{\beta} f(y,t)|}{|(x,t) - (y,t)|} \cdot |(x,t) - (y,t)| \right] 
\\&\le \frac{1}{N^{\frac{1}{d+1}}} \|f\|_{W^{k+1, \infty}([0,T] \times \bar{\Omega})} ,
\end{aligned}
\ee for any $|\alpha| \le k$, which yields the desired estimate \eqref{universal}.
\end{proof}

\begin{Thm} Let $T>0$ be an arbitrary positive time. Let $\psi_0 \in \bigcap\limits_{m \ge 0} \mathcal{D}(\l^m)$ and $\psi$ be the corresponding smooth solution to \eqref{fracheat} with homogeneous Dirichlet boundary conditions. Let $\delta > 0$.  Then there exists a function $\widehat{\psi}$ represented by a deep neural network with complexity $N$ such that $\mathcal{E}_G \le \delta$.
\end{Thm}

\begin{proof}
Let $\ell$ and $k$ be nonnegative integers. In view  of Proposition \ref{existenceneural}, there exists a neural network $\widehat{\psi}$ such that 
\be \label{hjhj}
\|\psi - \widehat{\psi}\|_{W^{\max\left\{2+\ell, k+1 \right\}, \infty}{([0,T]\times \overline{\Omega})}} \le CN^{-\frac{1}{d+1}},
\ee where $C$ is a constant depending on the initial data. The choice of the Sobolev space $W^{\max\left\{2+\ell, k+1 \right\}, \infty}{([0,T]\times \overline{\Omega})}$ is motivated by the estimates below. 

  {\bf{Step 1. Estimates for $\mathcal{E}_G^i.$}} We start by rewriting the PDE residual as follows, 
  \be 
\begin{aligned}
\mathcal{R}_i 
&= \pa_t (J_{\epsilon}\widehat{\psi} - \psi) 
+ u \cdot \na (J_{\epsilon}\widehat{\psi} - \psi) 
+ \l^{\alpha} (J_{\epsilon}\widehat{\psi} - \psi) 
\\&= \pa_t J_{\epsilon} (\widehat{\psi}-\psi) 
+ \pa_t (J_{\epsilon}\psi - \psi)
+ u \cdot \na J_{\epsilon}  (\widehat{\psi}-\psi)
+ u \cdot \na (J_{\epsilon}\psi - \psi)
\\&\quad\quad\quad\quad+ \l^{\alpha} J_{\epsilon}  (\widehat{\psi}-\psi)
+ \l^{\alpha} (J_{\epsilon}\psi - \psi). 
\end{aligned}
  \ee Fix an integer $j$ representing the order of the time derivative, such that $0 \le j \le k$. If $\ell$ is even, we apply $(-\Delta)^{\frac{\ell}{2}} \pa_t^{(j)}$ to $\mathcal{R}_i$, and we estimate each term separately.  If $\ell$ is odd, we apply $\na (-\Delta)^{\frac{\ell-1}{2}}\pa_t^{(j)}$ to $\mathcal{R}_i$ instead.

Suppose that $\ell$ is even. We first estimate the time derivative terms. Indeed, in view of Proposition \ref{lapcom}, $J_{\epsilon}$ and $(-\Delta)^{\frac{\ell}{2}}$ commute, and thus we have 
\be 
\beg{aligned}
\|(-\Delta)^{\frac{\ell}{2}} \pa_t^{(j)} \pa_t J_{\epsilon} (\widehat{\psi}-\psi)\|_{L^2}
= \|J_{\epsilon}(-\Delta)^{\frac{\ell}{2}} \pa_t^{(j+1)} (\widehat{\psi} -\psi)\|_{L^2},
\end{aligned}
\ee which can be bounded by 
\be 
\begin{aligned}
\|(-\Delta)^{\frac{\ell}{2}} \pa_t^{(j)} \pa_t J_{\epsilon} (\widehat{\psi}-\psi)\|_{L^2}
&\le C\|(-\Delta)^{\frac{\ell}{2}} \pa_t^{(j+1)} (\widehat{\psi} -\psi)\|_{L^2}
\le C\|\widehat{\psi}-\psi\|_{H^{\beta}([0,T]\times \overline{\Omega})}
\\&\le C\|\widehat{\psi}-\psi\|_{W^{\beta,\infty}([0,T]\times \overline{\Omega})}
\le CN^{-\frac{1}{d+1}},
\end{aligned}
\ee due to the uniform-in-$\epsilon$ boundedness of $J_{\epsilon}$ in $L^2$ and the estimate \eqref{hjhj}. Here $\beta = \max \left\{\ell, k+1 \right\}$. In view of the quantitative convergence bound obtained in Proposition \ref{quantitativeconv}, it follows that 
\be 
\begin{aligned}
&\|(-\Delta)^{\frac{\ell}{2}} \pa_t^{(j)} \pa_t (J_{\epsilon}{\psi} - \psi)\|_{L^2}^2
= \|(-\Delta)^{\frac{\ell}{2}} (J_{\epsilon}\pa_t^{(j+1)}\psi - \pa_t^{(j+1)}\psi)\|_{L^2}^2
\\&\le \kappa(\epsilon)^2 \|\pa_t^{(j+1)}\psi\|_{L^2} \|\l^{2\ell +d} \pa_t^{(j+1)} \psi\|_{L^2}
\le C\kappa(\epsilon)^2.
\end{aligned}
\ee Now we estimate the diffusion terms involving the fractional powers of the Dirichlet Laplacian. Since $\l^{\alpha}$ and $\pa_t^{(j)}$ commutes and $J_{\epsilon}$ and $\pa_t^{(j)}$ commutes, it holds that 
\be \la{iiii}
\|(-\Delta)^{\frac{\ell}{2}} \pa_t^{(j)} \l^{\alpha} J_{\epsilon}(\widehat{\psi}-\psi)\|_{L^2}
= \|\l^{\ell + \alpha} J_{\epsilon} (\pa_t^{(j)} \widehat{\psi} - \pa_t^{(j)} \psi)\|_{L^2}.
\ee Let $\gamma$ be the smallest integer greater than or equal to $\frac{\ell + \alpha}{2}$. Due to the continuous embedding of $\mathcal{D}(\l^{2\gamma})$ in $\mathcal{D}(\l^{\ell + \alpha})$, the fact that $J_{\epsilon}$ commutes with integer powers of the Laplacian, and the uniform-in-$\epsilon$ boundedness of $J_{\epsilon}$ in $L^2$, the identity \eqref{iiii} yields
\be
\begin{aligned}
&\|(-\Delta)^{\frac{\ell}{2}} \pa_t^{(j)} \l^{\alpha} J_{\epsilon}(\widehat{\psi}-\psi)\|_{L^2}
\le C\|(-\Delta)^{\gamma} J_{\epsilon} (\pa_t^{(j)}\widehat{\psi} - \pa_t^{(j)}\psi)\|_{L^2}
\\&= \|J_{\epsilon} (-\Delta)^{\gamma} (\pa_t^{(j)}\widehat{\psi} - \pa_t^{(j)}\psi)\|_{L^2}
\le C\|(-\Delta)^{\gamma} \pa_t^{(j)} (\widehat{\psi}- \psi)\|_{L^2}
\\&\le C\|\widehat{\psi}-\psi\|_{H^{\max\left\{\ell+2; k\right\}}([0,T]\times \overline{\Omega})}
\le \|\widehat{\psi}-\psi\|_{W^{\max\left\{\ell+2; k\right\},\infty}([0,T]\times \overline{\Omega})}
\le CN^{-\frac{1}{d+1}}.
\end{aligned}
\ee  We point out that $J_{\epsilon}$ and $\l^{\ell + \alpha}$ do not  commute when $\ell+ \alpha$ is not an integer because $\pa_t^{(i)}\widehat{\psi}$ and its higher-order Laplacians do not necesarily vanish on the boundary of $\Omega$, but $J_{\epsilon}$ and $(-\Delta)^{\gamma}$ commute, as shown in Proposition \ref{lapcom}, because $\gamma$ is an integer. As for the diffusion term that does not depend on the neural network, we have
\be 
\begin{aligned}
&\|(-\Delta)^{\frac{\ell}{2}} \pa_t^{(j)} \l^{\alpha}(J_{\epsilon}\psi -\psi)\|_{L^2}^2
= \|\l^{\ell + \alpha} (J_{\epsilon}\pa_t^{(j)} \psi - \pa_t^{(j)} \psi)\|_{L^2}^2
\\&\le \kappa(\epsilon)^2 \|\pa_t^{(j)}\psi\|_{L^2} \|\l^{2\ell + 2\alpha+d}\pa_t^{(j)}\psi \|_{L^2}
\le C\kappa(\epsilon)^2,
\end{aligned}
\ee due to the convergence property studied in Proposition \ref{quantitativeconv}. Finally, we address the advection terms involving $u$. By the Leibniz Theorem, we have 
\be \la{iuuu}
\begin{aligned}
&\|(-\Delta)^{\frac{\ell}{2}}\pa_t^{(j)}(u \cdot \na J_{\epsilon} (\widehat{\psi}-\psi))\|_{L^2}
\\&\quad\quad= \sum\limits_{v=0}^{j} \binom{j}{v}\|(-\Delta)^{\frac{\ell}{2}}\left(\pa_t^{(v)}u \cdot \na \pa_t^{(j-v)}J_{\epsilon}(\widehat{\psi}-\psi) \right)\|_{L^2}
\\&\quad\quad\le C\sum\limits_{v=0}^{j} \|\pa_t^{(v)}u\|_{W^{\ell,\infty}(\Omega)} \|\pa_t^{(j-v)}J_{\epsilon}(\widehat{\psi}-\psi) \|_{H^{1+\ell}(\Omega)}.
\end{aligned}
\ee Since $\mathcal{D}(\l^{2+\ell})$ is continuously embedded in $H^{1+\ell}(\Omega)$,  $J_{\epsilon}$ and $(-\Delta)^{1+\frac{\ell}{2}}$ commute (because $1+\frac{\ell}{2}$ is an integer), and $J_{\epsilon}$ is bounded in $L^2$, we estimate
\be
\begin{aligned}
&\|\pa_t^{(j-v)} J_{\epsilon}(\widehat{\psi}-\psi)\|_{H^{1+\ell}(\Omega)}
\le C\|(-\Delta)^{1+\frac{\ell}{2}} J_{\epsilon}\pa_t^{(j-v)} (\widehat{\psi}-\psi))\|_{L^2(\Omega)}
\\&\le C\|(-\Delta)^{1+\frac{\ell}{2}}\pa_t^{(j-v)} (\widehat{\psi}-\psi)\|_{L^2}^2
\le C\|\widehat{\psi}-\psi\|_{W^{\max\left\{2+\ell, j\right\},\infty}([0,T]\times \overline{\Omega})}
\le CN^{-\frac{1}{d+1}}. 
\end{aligned}
\ee Putting the latter back in \eqref{iuuu}, we infer that 
\be 
\|(-\Delta)^{\frac{\ell}{2}} \pa_t^{(j)} (u \cdot \na J_{\epsilon}(\widehat{\psi}-\psi))\|_{L^2}
\le C\|u\|_{W^{\max\left\{\ell,k\right\},\infty}([0,T]\times \Omega)}N^{-\frac{1}{d+1}}.
\ee As for the last term in $u$, we reapply the Leibniz rule and make use of the convergence property \eqref{quantitativeconv} to deduce that 
\be 
\begin{aligned}
&\|(-\Delta)^{\frac{\ell}{2}} \pa_t^{(j)} (u \cdot \na (J_{\epsilon}\psi - \psi))\|_{L^2}^2
\\&\quad\quad\le C\|u\|_{W^{\max\left\{\ell, k\right\},\infty}([0,T]\times \Omega)}^2 \sum\limits_{v=0}^{j} \kappa(\epsilon)^2\|\pa_t^{(j-v)}\psi\|_{L^2}\|\l^{2\ell+2+d}\pa_t^{(j-v)}\psi \|_{L^2}
\\&\quad\quad\le C\|u\|_{W^{\max\left\{\ell, k\right\},\infty}([0,T]\times \Omega)}^2\kappa(\epsilon)^2.
\end{aligned}
\ee This shows that 
\be 
\mathcal{E}_G^i
\le C\left(1+ \|u\|_{W^{\max \left\{\ell,k\right\},\infty}([0,T]\times \Omega)}^2\right) (N^{-\frac{2}{d+1}} + \kappa(\epsilon)^2),
\ee when $\ell$ is even.

Suppose $\ell$ is odd. In view of the continuous embeddings of $\mathcal{D}(\l)$ in $H^1(\Omega)$ and $\mathcal{D}({\l^2})$ in $\mathcal{D}(\l)$, the fact that $J_{\epsilon}$ commutes with $(-\Delta)^{\frac{\ell+1}{2}}$ (because $\frac{\ell+1}{2}$ is an integer), the boundedness of $J_{\epsilon}$ in $L^2$,  we can estimate the time derivative term as follows,
\be 
\begin{aligned}
&\|\na (-\Delta)^{\frac{\ell-1}{2}} \pa_t^{(j)}\pa_t J_{\epsilon} (\widehat{\psi}-\psi)\|_{L^2}
\le C\|\l \l^{\ell -1} \pa_t^{(j+1)} J_{\epsilon} (\widehat{\psi}-\psi)\|_{L^2}
\\&\le C\|\l^2 \l^{\ell -1} \pa_t^{(j+1)} J_{\epsilon} (\widehat{\psi}-\psi)\|_{L^2}
\le C\|(-\Delta)^{\frac{\ell+1}{2}} \pa_t^{(j+1)} (\widehat{\psi}-\psi)\|_{L^2}
\\&\le \|\widehat{\psi} -\psi\|_{W^{\max\left\{\ell+1, k+1\right\}, \infty}([0,T]\times \Omega)}
\le CN^{\frac{1}{d+1}}.
\end{aligned}
\ee As shown in the even case, it holds that 
\be 
\|\na (-\Delta)^{\frac{\ell-1}{2}}\pa_t^{(j)}\pa_t (J_{\epsilon} \psi - \psi)\|_{L^2}^2
\le C\|\l^{\ell} \pa_t^{(j+1)}(J_{\epsilon}\psi - \psi)\|_{L^2}^2
\le C\kappa(\epsilon)^2. 
\ee As for the remaining terms, we employ the same idea of embedding $\mathcal{D}(\l)$ in $H^1(\Omega)$ and follow verbatim the computations implemented in the even case. We omit these details to avoid redundancy. Therefore, we obtain 
\be 
\mathcal{E}_G^i
\le C\left(1+ \|u\|_{W^{\max \left\{\ell,k\right\},\infty}([0,T]\times \Omega)}^2\right) (N^{-\frac{2}{d+1}} + \kappa(\epsilon)^2),
\ee when $\ell$ is odd.

 {\bf{Step 2. Estimates for $\mathcal{E}_G^{t}.$}} We seek good control of $\mathcal{E}_G^{t}[\ell,k;\theta]$ when $k$ is even. The case where $k$ is odd is similar and will be omitted. 
 When $k$ is even, we have 
 \be 
\begin{aligned}
\mathcal{E}_G^t
&\le C\sum\limits_{j=1}^{k} \|(-\Delta)^{\frac{\ell}{2}} (\pa_t^{(j)} (\widehat{\psi}-\psi))(x,0) \|_{L^2(\Omega)}^2
\le C \sum\limits_{j=1}^{k} \|\pa_t^{(j)} (\widehat{\psi}-\psi)(x,0)\|_{H^{\ell}(\Omega)}^2
\\&\le C\sum\limits_{j=1}^{k}\|\pa_t^{(j)}(\widehat{\psi}-\psi)\|_{H^{\ell+1}([0,T] \times \overline{\Omega})}^2
\le C\|\widehat{\psi}-\psi\|_{W^{\max\left\{\ell+1, k\right\},\infty}([0,T]\times \overline{\Omega})}^2 
\le CN^{-\frac{2}{d+1}},
\end{aligned}
\ee by the trace theorem. 
 
  {\bf{Step 3. Estimates for $\mathcal{E}_G^b.$}} By making use of the trace theorem and the fact that the solution $\psi$ vanishes on $\pa \Omega$, we have 
  \be 
\begin{aligned}
\mathcal{E}_G^b \le \|\widehat{\psi}-\psi\|_{L^2([0,T] \times \pa \Omega)}^2
\le C\|\widehat{\psi} -\psi\|_{H^1([0,T] \times \overline{\Omega})}^2
\le C\|\widehat{\psi} - \psi\|_{W^{1,\infty}([0,T]\times \overline{\Omega})}^2 \le CN^{-\frac{2}{d+1}}. 
\end{aligned}
\ee 
  
   {\bf{Step 4. Conclusion.}} Putting all these estimates together, we infer that 
   \be 
\mathcal{E}_G 
\le C\left(1+ \|u\|_{W^{\max \left\{\ell,k,1\right\},\infty}([0,T]\times \Omega)}^2\right) (N^{-\frac{2}{d+1}} + \kappa(\epsilon)^2).
   \ee 
Finally, we choose $N$ sufficiently large and $\epsilon$ sufficiently small to deduce that 
$ 
\mathcal{E}_G \le \delta.$ 
\end{proof}

\section{Total Error Estimates}\label{sec:TEE}
In this section, we derive bounds for the total error between the true solution of the fractional diffusion equation and its neural network approximation. The total error accounts for both the generalization error and the consistency of the regularization strategy introduced earlier. A key step in this analysis is to control the approximation error of the mollified neural network output in Sobolev norms. The following proposition establishes that the error in $H^k$ norm can be bounded in terms of the modified energy functional defined in Section~\ref{sec:regularizer}. This sets the stage for our main result, which shows that the total error can be controlled by the generalization error derived in Section~\ref{sec:GEE}, thereby linking training performance to the overall accuracy of the PINN approximation.

\beg{prop} \la{dom} Let $\epsilon \in (0,1)$. Let $\psi$ be the solution to \eqref{fracheat} and $\psi_{\theta}$ be a neural network approximating $\psi$. Then it holds that 
\be 
\|\psi - J_{\epsilon} \psi_{\theta}\|_{H^k([0,T] \times \Omega)}^2 \le C\mathcal{E}[k,k;\theta]^2,
\ee for some universal constant $C >0 $.
\end{prop}

\begin{proof} In view of the continuous embedding of $\mathcal{D}(\l^j)$ into $H^j(\Omega)$ and the fact that $\pa_t^{(j)}(\psi - J_{\epsilon}\psi_{\theta})$ vanishes on $\pa \Omega$, we have 
\be 
\beg{aligned}
\|\psi - J_{\epsilon} \psi_{\theta}\|_{H^k([0,T]\times \Omega)}^2
&= \sum\limits_{i=0}^{k} \sum\limits_{j=0}^{k} \|D^j \pa_t^{(i)} (\psi - J_{\epsilon}\psi_{\theta})\|_{L^2([0,T] \times \Omega)}^2
\\&\le C\sum\limits_{i=0}^{k} \sum\limits_{j=0}^{k} \|\l^j \pa_t^{(i)} (\psi - J_{\epsilon}\psi_{\theta})\|_{L^2([0,T] \times \Omega)}^2
\\&\le  C\sum\limits_{i=0}^{k} \|\l^k \pa_t^{(i)} (\psi - J_{\epsilon}\psi_{\theta})\|_{L^2([0,T] \times \Omega)}^2
\\&\le C\mathcal{E}[k,k;\theta]^2.
\end{aligned}
\ee 

\end{proof}
\beg{Thm} Let $T>0$ be an arbitrary positive time. Let $\psi_0 \in \bigcap\limits_{m \ge 0} \mathcal{D}(\l^m)$ and $\psi$ be the corresponding smooth solution to \eqref{fracheat} with homogeneous Dirichlet boundary conditions. Let $\widehat{\psi}$ be a neural network approximating $\psi$. Let $\ell$ and $k$ be nonnegative integers. Then it holds that  
\be 
\mathcal{E}[\ell, k;\theta]^2 \le C[u] \mathcal{E}_G[\ell+k, k;\theta]^2,
\ee for some constant $C[u]$ depending on $T, \ell, k$ and the $W^{\max\left\{\ell, k\right\}, \infty}$ norm of $u$. 
\end{Thm}

\begin{proof}
    We implement a proof by induction on $k$. For the base step $(k=0)$, the total error is given by 
\be 
\mathcal{E}[\ell,0;\theta] = \int_{0}^{T}\|\l^{\ell} (\psi - J_{\epsilon} \widehat{\psi})\|_{L^2}^2 dt.
\ee Subtracting the residual $\mathcal{R}_i:=\mathcal{R}_i[\theta]$ from the PDE obeyed by $\psi$, we have 
\be \la{tot1}
\pa_t (\psi - J_{\epsilon} \widehat{\psi}) + u \cdot \na (\psi - J_{\epsilon}\widehat{\psi}) + \l^{\alpha}(\psi - J_{\epsilon}\widehat\psi) = - \mathcal{R}_i.
\ee We point out that $\mathcal{R}_i \notin \mathcal{D}(\l^{\ell})$ because it does not necessarily vanish on $\pa \Omega$. Thus we cannot apply $\l^{\ell}$ to the equation \eqref{tot1} obeyed by $\psi - J_{\epsilon} \widehat{\psi}$. The remedy is to use the fact that $\l^{\ell} = (-\Delta_D)^{\frac{\ell}{2}}$ which is local when $\ell$ is even, and the fact that $\l^{\ell}  = \l \l^{\ell - 1}$ amounts to $\na (-\Delta_D)^{\frac{\ell -1}{2}} $ in $L^2$ when $\ell$ is odd. Since $\na$ and $\Delta$ are local operators, we can apply them to $\mathcal{R}_i$ as long as $\mathcal{R}_i$ is smooth enough. We use this trick to address the following two cases:

{\bf{Case 1. $\ell$ is even.}} We apply $(-\Delta)^{\frac{\ell}{2}}$ to \eqref{tot1} and obtain 
\be 
\pa_t (-\Delta)^{\frac{\ell}{2}} (\psi - J_{\epsilon} \widehat{\psi})
+ (-\Delta)^{\frac{\ell}{2}}\l^{\alpha} (\psi - J_{\epsilon}\widehat{\psi}) + (-\Delta)^{\frac{\ell}{2}}(u \cdot \na (\psi - J_{\epsilon}\widehat{\psi})) = - (-\Delta)^{\frac{\ell}{2}} \mathcal{R}_i.
\ee We multiply the latter by $(-\Delta)^{\frac{\ell}{2}}$ and integrate over $\Omega$. Since $\psi - J_{\epsilon}\widehat{\psi} \in \mathcal{D}(\l^{\ell})$ and $\l^{\alpha} (\psi - J_{\epsilon}\widehat{\psi}) \in \mathcal{D}(\l^{\ell})$, it holds that 
\be 
(-\Delta)^{\frac{\ell}{2}} (\psi- J_{\epsilon}\widehat{\psi}) = (-\Delta_D)^{\frac{\ell}{2}} (\psi - J_{\epsilon} \widehat{\psi}) = \l^{\ell} (\psi - J_{\epsilon} \widehat{\psi}),
\ee  and 
\be 
(-\Delta)^{\frac{\ell}{2}} \l^{\alpha}(\psi - J_{\epsilon}\widehat{\psi}) = (-\Delta_D)^{\frac{\ell}{2}} \l^{\alpha} (\psi - J_{\epsilon}\widehat{\psi})
= \l^{\ell} \l^{\alpha} (\psi - J_{\epsilon}\widehat{\psi}),
\ee yielding
\be 
\begin{aligned}
&\frac{1}{2} \frac{d}{dt} \|\l^{\ell} (\psi - J_{\epsilon}\widehat{\psi})\|_{L^2}^2 + \|\l^{\ell + \frac{\alpha}{2}} (\psi - J_{\epsilon} \widehat{\psi})\|_{L^2}^2 
\\&\quad\quad= - \int_{\Omega} (-\Delta)^{\frac{\ell}{2}} (u \cdot \na (\psi - J_{\epsilon}\widehat{\psi})) (-\Delta)^{\frac{\ell}{2}} (\psi - J_{\epsilon}\widehat{\psi}) - \int_{\Omega} (-\Delta)^{\frac{\ell}{2}}\mathcal{R}_i (-\Delta)^{\frac{\ell}{2}} (\psi - J_{\epsilon} \widehat{\psi}). 
\end{aligned}
\ee Denoting the commutator 
\be  
[(-\Delta)^{\frac{\ell}{2}}, u \cdot \na] (\psi - J_{\epsilon}\widehat{\psi})
:= (-\Delta)^{\frac{\ell}{2}} (u \cdot \na (\psi- J_{\epsilon}\widehat{\psi})) - u \cdot \na (-\Delta)^{\frac{\ell}{2}} (\psi - J_{\epsilon} \widehat{\psi}),
\ee and using the cancellation law 
\be 
\int_{\Omega} u \cdot \na (-\Delta)^{\frac{\ell}{2}} (\psi - J_{\epsilon}\widehat{\psi}) \cdot (-\Delta)^{\frac{\ell}{2}} (\psi - J_{\epsilon}\widehat{\psi}) = 0
\ee that holds due to the divergence-free condition obeyed by $u$, we can write the velocity term as follows,
\be 
\begin{aligned}
\mathcal{U}
&:= - \int_{\Omega}(-\Delta)^{\frac{\ell}{2}} (u \cdot \na (\psi - J_{\epsilon}\widehat{\psi})) (-\Delta)^{\frac{\ell}{2}} (\psi - J_{\epsilon}\widehat{\psi}) 
\\&= - \int_{\Omega} [(-\Delta)^{\frac{\ell}{2}}, u \cdot \na] (\psi - J_{\epsilon}\widehat{\psi}) (-\Delta)^{\frac{\ell}{2}} (\psi - J_{\epsilon}\widehat{\psi}).
\end{aligned}
\ee Using the commutator estimate 
\be 
\|[(-\Delta)^{\frac{\ell}{2}}, u \cdot \na] (\psi - J_{\epsilon}\widehat{\psi})\|_{L^2}
\le C\|u\|_{W^{\ell,\infty}} \|\psi - J_{\epsilon} \widehat{\psi}\|_{H^{\ell}},
\ee and the elliptic regularity estimate 
\be 
\|\psi - J_{\epsilon} \widehat{\psi}\|_{H^{\ell}}
\le C\|\l^{\ell}(\psi - J_{\epsilon} \widehat{\psi})  \|_{L^2}
\ee that holds due to the continuous embedding of $\mathcal{D}(\l^{\ell})$ in $H^{\ell}$, we infer that 
\be 
\mathcal{U} 
\le C\|u\|_{W^{\ell,\infty}} \|\l^{\ell}(\psi - J_{\epsilon} \widehat{\psi})\|_{L^2}^2.
\ee Thus, we obtain the differential inequality
\be 
\begin{aligned}
& \frac{d}{dt} \|\l^{\ell}(\psi - J_{\epsilon}\widehat{\psi})\|_{L^2}^2 + \|\l^{\ell + \frac{\alpha}{2}}(\psi - J_{\epsilon} \widehat{\psi})\|_{L^2}^2 
\\&\quad\quad\le  C\|(-\Delta)^{\frac{\ell}{2}} \mathcal{R}_i\|_{L^2}^2 + C(1+\|u\|_{W^{\ell,\infty}}) \|\l^{\ell}(\psi - J_{\epsilon}\widehat{\psi})\|_{L^2}^2.
\end{aligned}
\ee By Gronwall's inequality, we deduce that 
\be 
\mathcal{E}[\ell,0;\theta]^2 
\le C_T[u] \mathcal{E}_G[\ell, 0;\theta]^2,
\ee where $C_T[u]$ is a constant depending only on $T$ and the $W^{\ell,\infty}$ norm of $u$.

{\bf{Case 2. $\ell$ is odd.}} Applying $\na (-\Delta)^{\frac{\ell -1}{2}}$ to \eqref{tot1}, we have 
\be 
\begin{aligned}
&\pa_t \na (-\Delta)^{\frac{\ell -1}{2}} (\psi - J_{\epsilon}\widehat{\psi}) 
+ \na (-\Delta)^{\frac{\ell -1}{2}} \l^{\alpha} (\psi - J_{\epsilon}\widehat{\psi}) 
\\&\quad\quad+ \na (-\Delta)^{\frac{\ell -1}{2}} (u \cdot \na (\psi - J_{\epsilon}\widehat{\psi}) ) = - \na (-\Delta)^{\frac{\ell -1}{2}} \mathcal{R}_i.
\end{aligned}
\ee Taking the $L^2$ inner product of the latter with $\na (-\Delta)^{\frac{\ell -1}{2}}(\psi - J_{\epsilon}\widehat{\psi})$ gives
\be 
\begin{aligned}
&\frac{1}{2} \frac{d}{dt} \|\na (-\Delta)^{\frac{\ell -1}{2}} (\psi - J_{\epsilon}\widehat{\psi})  \|_{L^2}^2 
+ \int_{\Omega} \na (-\Delta)^{\frac{\ell -1}{2}} \l^{\alpha} (\psi - J_{\epsilon}\widehat{\psi})  \cdot \na (-\Delta)^{\frac{\ell -1}{2}} (\psi - J_{\epsilon}\widehat{\psi}) 
\\&\quad\quad= -\int_{\Omega} \na (-\Delta)^{\frac{\ell-1}{2}} (u \cdot \na (\psi - J_{\epsilon}\widehat{\psi})) \cdot \na (-\Delta)^{\frac{\ell -1}{2}} (\psi - J_{\epsilon}\widehat{\psi}) 
\\&\quad\quad\quad\quad-\int_{\Omega} \na (-\Delta)^{\frac{\ell -1}{2}}\mathcal{R}_i \cdot \na (-\Delta)^{\frac{\ell-1}{2}} (\psi - J_{\epsilon}\widehat{\psi}).
\end{aligned}
\ee Since $\psi - J_{\epsilon}\widehat{\psi} \in \mathcal{D}(\l^{\ell})$, we have 
\be 
\frac{d}{dt} \|\na (-\Delta)^{\frac{\ell-1}{2}} (\psi - J_{\epsilon} \widehat{\psi})\|_{L^2}^2
= \frac{d}{dt} \|\l \l^{\ell -1}(\psi - J_{\epsilon}\widehat{\psi}) \|_{L^2}^2 = \frac{d}{dt} \|\l^{\ell}(\psi - J_{\epsilon}\widehat{\psi}) \|_{L^2}^2.
\ee Since $\psi - J_{\epsilon}\widehat{\psi} \in \mathcal{D}(\l^k)$ for any $k \in \N$, we have $(-\Delta)^{\frac{\ell-1}{2}}\l^{\alpha}(\psi - J_{\epsilon}\widehat{\psi}) |_{\pa \Omega} = 0$, and thus we can integrate by parts to obtain the identity 
\be 
\begin{aligned}
& \int_{\Omega} \na (-\Delta)^{\frac{\ell -1}{2}} \l^{\alpha} (\psi - J_{\epsilon}\widehat{\psi})  \cdot \na (-\Delta)^{\frac{\ell -1}{2}} (\psi - J_{\epsilon}\widehat{\psi})
\\&\quad\quad= - \int_{\Omega} (-\Delta)^{\frac{\ell-1}{2}} \l^{\alpha} (\psi - J_{\epsilon}\widehat{\psi})  \Delta (-\Delta)^{\frac{\ell-1}{2}} (\psi - J_{\epsilon}\widehat{\psi}) 
\\&\quad\quad= \int_{\Omega} \l^{\ell -1} \l^{\alpha} (\psi - J_{\epsilon}\widehat{\psi}) \l^{2}\l^{{\ell-1}}(\psi - J_{\epsilon}\widehat{\psi}) 
\\&\quad\quad=\|\l^{\ell + \frac{\alpha}{2}}(\psi - J_{\epsilon}\widehat{\psi}) \|_{L^2}^2.
\end{aligned}
\ee As for the $u$-term, we employ a similar technique to the case when $\ell$ is even and make use of commutators to derive the following estimates,
\be 
\begin{aligned}
&-\int_{\Omega} \na (-\Delta)^{\frac{\ell-1}{2}} (u \cdot \na (\psi - J_{\epsilon}\widehat{\psi})) \cdot \na (-\Delta)^{\frac{\ell -1}{2}} (\psi - J_{\epsilon}\widehat{\psi}) 
\\&\quad\quad= - \int_{\Omega} [\na (-\Delta)^{\frac{\ell-1}{2}}, u \cdot \na]  (\psi - J_{\epsilon}\widehat{\psi}) \cdot \na (-\Delta)^{\frac{\ell -1}{2}}(\psi - J_{\epsilon}\widehat{\psi}) 
\\&\quad\quad\le C\|u\|_{W^{\ell,\infty}} \|\psi - J_{\epsilon}\widehat{\psi}\|_{H^{\ell}}^2 \le C\|u\|_{H^{\ell}} \|\l^{\ell}(\psi - J_{\epsilon}\widehat{\psi})\|_{L^2}^2.
\end{aligned}
\ee Combining these estimates and applying Gronwall's inequality, we infer that 
\be 
\mathcal{E}[\ell,0;\theta]^2 
\le C_T[u]\mathcal{E}_G[\ell, 0;\theta]^2,
\ee when $\ell$ is odd. Here $C_T[u]$ is a constant depending only on $T$ and the $W^{\ell,\infty}$ norm of $u$. Now we assume that there exists a constant $C_{T,k-1}[u]$ that depends only on $T$, $k$, and $\|u\|_{W^{\max\left\{\ell,k-1\right\}, \infty}([0,T]\times \Omega)}^2$  such that 
\be 
\mathcal{E}[\ell, k-1;\theta]^2
\le C_{T,k-1}[u] \mathcal{E}_G[\ell+k-1, k-1;\theta]^2,
\ee for all $\ell \in \N$. We prove that 
\be 
\mathcal{E}[\ell, k;\theta]^2 \le C_{T, k}[u] \mathcal{E}_G[\ell+k, k;\theta]^2
\ee holds for all $\ell \in \N$. To this end, we fix $\ell \in \N$ and distinguish two cases:

{\bf{Case 1. $\ell$ is even.}} In this case, we have 
\be 
\begin{aligned}
&\pa_t^{(k+1)} (-\Delta)^{\frac{\ell}{2}} (\psi - J_{\epsilon} \widehat{\psi})
+ \pa_t^{(k)} (-\Delta)^{\frac{\ell}{2}} \Lambda^{\alpha}(\psi - J_{\epsilon} \widehat{\psi})
\\&\quad\quad= -\pa_t^{(k)} (-\Delta)^{\frac{\ell}{2}} \mathcal{R}_i
- \pa_t^{(k)} (-\Delta)^{\frac{\ell}{2}} (u \cdot \na (\psi - J_{\epsilon} \widehat{\psi})).
\end{aligned}
\ee We point out that $\pa_t^{(k)}, \na$, and $(-\Delta)^{\frac{\ell}{2}}$ are local operators and commute with each other. Multiplying the latter by $\pa_t^{(k)} (-\Delta)^{\frac{\ell}{2}} (\psi - J_{\epsilon} \widehat{\psi})$ and integrating over $\Omega$ produce  
\be 
\begin{aligned}
&\frac{1}{2} \frac{d}{dt} \|\l^{\ell} \pa_t^{(k)} (\psi - J_{\epsilon} \widehat{\psi})\|_{L^2}^2
+ \|\l^{\ell + \frac{\alpha}{2}} \pa_t^{(k)} (\psi - J_{\epsilon} \widehat{\psi})\|_{L^2}^2
\\&\quad\quad= -\int_{\Omega} \pa_t^{(k)} (-\Delta)^{\frac{\ell}{2}} \mathcal{R}_i \pa_t^{(k)} (-\Delta)^{\frac{\ell}{2}} (\psi - J_{\epsilon} \widehat{\psi})
\\&\quad\quad\quad\quad-\int_{\Omega} \pa_t^{(k)} (-\Delta)^{\frac{\ell}{2}} (u \cdot \na (\psi - J_{\epsilon} \widehat{\psi})) \pa_t^{(k)} (-\Delta)^{\frac{\ell}{2}}(\psi - J_{\epsilon} \widehat{\psi}). 
\end{aligned}
\ee By the Leibnitz Theorem, we decompose the $u$-term as follows,
\be 
\begin{aligned}
&-\int_{\Omega} \pa_t^{(k)} (-\Delta)^{\frac{\ell}{2}} (u \cdot \na (\psi - J_{\epsilon} \widehat{\psi})) \pa_t^{(k)} (-\Delta)^{\frac{\ell}{2}}(\psi - J_{\epsilon} \widehat{\psi})
\\&\quad\quad= -\sum\limits_{i=0}^{k} \binom{k}{i} \int_{\Omega} (-\Delta)^{\frac{\ell}{2}} \left[\pa_t^{(i)} u \cdot \na \pa_t^{(k-i)} (\psi - J_{\epsilon}\widehat{\psi}) \right] (-\Delta)^{\frac{\ell}{2}} \pa_t^{(k)} (\psi - J_{\epsilon}\widehat{\psi}) 
\\&\quad\quad= -\sum\limits_{i=1}^{k} \binom{k}{i} \int_{\Omega} (-\Delta)^{\frac{\ell}{2}} \left[\pa_t^{(i)} u \cdot \na \pa_t^{(k-i)} (\psi - J_{\epsilon}\widehat{\psi}) \right] (-\Delta)^{\frac{\ell}{2}} \pa_t^{(k)} (\psi - J_{\epsilon}\widehat{\psi}) 
\\&\quad\quad\quad\quad-\int_{\Omega} (-\Delta)^{\frac{\ell}{2}} \left[u \cdot \na \pa_t^{(k)} (\psi-J_{\epsilon}\widehat{\psi}) \right] (-\Delta)^{\frac{\ell}{2}} \pa_t^{(k)} (\psi- J_{\epsilon}\widehat{\psi})
\\&\quad\quad:= \mathcal{U}_1 + \mathcal{U}_2.
\end{aligned}
\ee 
Using Sobolev product estimates and the continuous embedding of $\mathcal{D}(\l^{m})$ into $H^m$ for any $m \in \N$, we estimate
\be 
\begin{aligned}
\mathcal{U}_1
&\le C\sum\limits_{i=1}^{k} \|\pa_t^{(i)} u\|_{W^{\ell, \infty}} \|\pa_t^{(k-i)} (\psi - J_{\epsilon} \widehat{\psi})\|_{H^{\ell + 1}}  \|\pa_t^{(k)} (\psi-J_{\epsilon}\widehat{\psi})\|_{H^{\ell}}
\\&\le C\sum\limits_{i=1}^{k} \|\pa_t^{(i)} u\|_{W^{\ell, \infty}}^2 \|\l^{\ell +1} \pa_t^{(k-i)} (\psi - J_{\epsilon} \widehat{\psi})\|_{L^{2}}^2 
+ \|\l^{\ell} \pa_t^{(k)} (\psi-J_{\epsilon}\widehat{\psi})\|_{L^2}^2. 
\end{aligned}
\ee We point out that 
\be 
\begin{aligned}
&\int_{0}^{T} \sum\limits_{i=1}^{k} \|\pa_t^{(i)} u\|_{W^{\ell, \infty}}^2 \|\l^{\ell +1} \pa_t^{(k-i)} (\psi - J_{\epsilon} \widehat{\psi})\|_{L^{2}}^2 dt
\\&\le C\|u\|_{W^{\max\left\{\ell,k\right\}, \infty}([0,T]\times \Omega)}^2 \int_{0}^{T}\sum\limits_{i=1}^{k}  \|\l^{\ell +1} \pa_t^{(k-i)} (\psi - J_{\epsilon} \widehat{\psi})\|_{L^{2}}^2 dt
\\&\le C\|u\|_{W^{\max\left\{\ell,k\right\}, \infty}([0,T]\times \Omega)}^2 \mathcal{E} [\ell+1, k-1;\theta]^2
\\&\le C\|u\|_{W^{\max\left\{\ell,k\right\}, \infty}([0,T]\times \Omega)}^2 \mathcal{E}_G [\ell+k, k-1;\theta]^2,
\end{aligned}
\ee by the induction hypothesis. 
As for the term $\mathcal{U}_2$, we make use of the cancellation law 
\be 
\int_{\Omega} u \cdot \na [(-\Delta)^{\frac{\ell}{2}} \pa_t^{k} (\psi - J_{\epsilon}\widehat{\psi})](-\Delta)^{\frac{\ell}{2}} \pa_t^{k} (\psi - J_{\epsilon}\widehat{\psi})  = 0,
\ee to write $\mathcal{U}_2$ in terms of the commutator $[(-\Delta)^{\frac{\ell}{2}}, u \cdot \na]$ as follows, 
\be 
\begin{aligned}
\mathcal{U}_2 &= - \int_{\Omega} [(-\Delta)^{\frac{\ell}{2}}, u \cdot \na] \pa_t^{(k)} (\psi - J_{\epsilon}\widehat{\psi}) (-\Delta)^{\frac{\ell}{2}} \pa_t^{(k)}(\psi - J_{\epsilon}\widehat{\psi}), 
\end{aligned}
\ee and then we make use of commutator estimates to bound $\mathcal{U}_2$ by 
\be 
\mathcal{U}_2 
\le C\|u\|_{W^{\ell, \infty}} \|\l^{\ell} \pa_t^{(k)} (\psi - J_{\epsilon} \widehat{\psi})\|_{L^2}^2.
\ee Putting all these estimates together and applying Gronwall's inequality yield
\be 
\mathcal{E}[\ell, k; \theta]^2
\le C_{T,k}[u] \mathcal{E}_G[\ell+k, k;\theta]^2,
\ee for some positive constant $C_{T,k}[u]$ depending only on $T, k, \ell$ and the $W^{\max\left\{\ell, k\right\},\infty}([0,T]\times \Omega)$ norm of $u$.

{\bf{Case 2. $\ell$ is odd.}} This case is similar to the previous cases. The proof will be omitted to avoid redundancy. 
\end{proof} 

\section{Concluding Remarks}\label{sec:conclusion}
In this paper, we established a convergence theory for PINNs applied to fractional diffusion equations posed on bounded domains with spectral Dirichlet boundary conditions. By introducing a spectrally-defined mollification strategy, we ensured boundary compatibility and derived rigorous error estimates in standard Sobolev norms. Our analysis is grounded in tools from PDE theory and functional analysis, rather than numerical discretization schemes, and contributes to the growing body of analytical work on PINNs, particularly in nonlocal settings where classical methods face structural challenges.

The mollification strategy employed in this paper allows approximation of solutions by neural networks in any space-time Sobolev space. Such approximations fail when cutoff functions are used instead. In fact, multiplying the neural network by cutoffs $\chi_{\epsilon}$ yields vanishing on the boundary but gives rise to the need for proper estimation of the differences $\Lambda_{D}^{\alpha} (\chi_{\epsilon} \psi - \psi)$ in Sobolev spaces. More precisely, one needs the distance between $\chi_{\epsilon} \psi$ and $\psi$ to be sufficiently small in any $\mathcal{D}(\l^j)$ space. This is obviously achievable in $L^2$ but not necessarily in $\mathcal{D}(\Lambda^{j})$. For instance, when $j =1$, it holds that
          \be
          \|\Lambda_{D} (\chi_{\epsilon} \psi - \psi)\|_{L^2} = \|\na ((\chi_{\epsilon} - 1)\psi)\|_{L^2}^2.
          \ee But when $\epsilon$ approaches 0, the gradient of $\chi_{\epsilon}$ blows up, and thus the existence of a small $\epsilon$ for which the latter norm is sufficiently small is not clear. In contrast, we know that $J_{\epsilon}\psi$ converges to $\psi$ in $\mathcal{D}(\Lambda^{j})$, prioritizing the use of $J_{\epsilon}$ over $\chi_{\epsilon}$.


In future work, it would be of interest to extend this framework to a wider class of partial differential equations involving some nonlinear aspects and different types of nonlocal operators. 


\section*{Acknowledgements}
L.C. was partially supported by the National Key R\&D Program of China No. 2021YFA1003001 and the NSFC grant No. 12271537. R.H. was partially supported by the ONR grant N00014-24-1-2432, the Simons Foundation (MP-TSM-00002783), and the NSF grant DMS-2420988.

\appendix

\section{Existence and Uniqueness of Global Smooth Solutions} \la{exandun}

\beg{Thm} \la{exiun} Let $\alpha \in [0,2]$. Let $T > 0$. Let $u$ be a smooth divergence-free vector field such that $u \cdot n|_{\pa \Omega} = 0$. Let $f$ be a given smooth function such that $(-\Delta)^k f|_{\pa \Omega} = 0$ for any $k \in \N$. Let $ \psi_0 \in \bigcap_{k=0}^{\infty} \mathcal{D}(\l^k)$. Then the advetion-diffusion equation 
\be \label{readif}
\pa_t \psi + u \cdot \na \psi + \l^{\alpha}\psi = f,
\ee 
equipped with homogeneous Dirichlet boundary conditions and initial data $\psi_0$ has a unique solution $\psi$ obeying
\be 
\psi \in L^{\infty}(0,T; \mathcal{D}(\l^k)),
\ee for all $k \in \N$.
\end{Thm}

\begin{proof}
For $\epsilon \in (0,1)$, we consider the regularized equations 
\be 
\pa_t \psi^{\epsilon} 
+ u \cdot \na \psi^{\epsilon}
+ \l^{\alpha} \psi^{\epsilon}
- \epsilon \Delta \psi^{\epsilon}
= f,
\ee with initial data 
\be 
\psi^{\epsilon}|_{\pa \Omega} = 0,
\ee and boundary conditions
\be 
\psi^{\epsilon}(x,0)= \psi_0(x).
\ee 
These systems have global smooth solutions 
\be 
\psi^{\epsilon} \in L^{\infty} (0,T; \mathcal{D}(\l^k)),
\ee for any $k \in \N$, a fact that follows from a classical Galerkin approximation scheme and passage to the limit via compactness arguments (see, for instance \cite{abdo2024regularity}). We establish bounds for solutions to these regularized systems that are independent of $\epsilon$ and deduce the existence of solutions 
\be 
\psi \in L^{\infty}(0,T; \mathcal{D}(\l^k)),
\ee for any $k \in \N$. In fact, the $\mathcal{D}(\l^k)$ norm of $\psi^{\epsilon}$  evolves according to 
\be 
\frac{1}{2} \frac{d}{dt}\|\l^k \psi^{\epsilon}\|_{L^2}^2
+ \|\l^{k+\frac{\alpha}{2}}\psi^{\epsilon}\|_{L^2}^2
= - \int_{\Omega} \l^k (u \cdot \na \psi^{\epsilon}) \l^k \psi^{\epsilon} dx 
+ \int_{\Omega} \l^k f \l^k \psi^{\epsilon} dx.
\ee We point out that $u \cdot \na \psi^{\epsilon} \in \mathcal{D}(\l^k)$ for all $k \in \N$ because 
\be 
u \cdot \na \psi^{\epsilon}
= -\pa_t \psi^{\epsilon} - \l^{\alpha}\psi^{\epsilon} 
+ \epsilon \Delta \psi^{\epsilon} + f,
\ee where the right-hand side belongs to $\mathcal{D}(\l^k)$ for all $k \in \N$. 

In order to estimate the term in $u$, we use the cancellation laws
\be \la{cance1}
\int_{\Omega} u \cdot \na \l^{k}\psi^{\epsilon}  \l^{k} \psi^{\epsilon} dx = 0,
\ee when $k$ is even, and 
\be \la{cance2}
\int_{\Omega} u \cdot \na \na \l^{k-1} \cdot \na \l^{k-1}\psi^{\epsilon}  dx = 0,
\ee when $k$ is odd, and we obtain 
\be 
- \int_{\Omega} \l^k (u \cdot \na \psi^{\epsilon}) \l^k \psi^{\epsilon} dx 
= -\int_{\Omega} [\l^k, u \cdot \na]\psi^{\epsilon} \l^{k}\psi^{\epsilon} dx 
\le C\|u\|_{W^{k,\infty}}\|\l^k \psi^{\epsilon}\|_{L^2}^2,
\ee when $k$ is even 
and 
\be 
- \int_{\Omega} \l^k (u \cdot \na \psi^{\epsilon}) \l^k \psi^{\epsilon} dx 
= -\int_{\Omega} [\na \l^{k-1}, u \cdot \na]\psi^{\epsilon} \na \l^{k-1}\psi^{\epsilon} dx 
\le C\|u\|_{W^{k,\infty}}\|\l^k \psi^{\epsilon}\|_{L^2}^2,
\ee when $k$ is odd. The last two estimates follow from expanding the commutators, applying H\"older's inequality, and using the identity $- \na \cdot \na = -\Delta = \l^2$ in the odd case. The cancellations \eqref{cance1} and \eqref{cance2} hold due to the divergence-free property obeyed by $u$ and the boundary assumption $u \cdot n|_{\pa \Omega} = 0$. The need for this property justifies the parabolic regularization scheme that preserves transport by divergence-free vector fields, in contrast to Galerkin approximations that destroy the aforementioned structure. 
Finally, we apply Gronwall's inequality and infer that 
\be 
\begin{aligned}
&\sup\limits_{0 \le t \le T} \|\l^k \psi^{\epsilon}(t)\|_{L^2}^2
\\&\quad\quad\le \left(\|\l^{k}\psi_0\|_{L^2}^2 + C \int_{0}^{T} \|\l^{k- \frac{\alpha}{2}}f\|_{L^2}^2 dt \right) \exp \left\{C\int_{0}^{T} \|u\|_{W^{k,\infty}} dt  \right\}.
\end{aligned}
\ee We omit further details. 
\end{proof}

\begin{rem}
We point out that the spatial smoothness of solutions derived in Theorem \ref{exiun} yields their time smoothness in view of the PDE \eqref{readif}. As a consequence, the unique solution $\psi$ is $C^{\infty}$ in both space and time. 
\end{rem}

\bibliographystyle{siamplain}
\bibliography{references}

@article{chen2025structure,
  title={Structure and asymptotic preserving deep neural surrogates for uncertainty quantification in multiscale kinetic equations},
  author={Chen, Wei and Dimarco, Giacomo and Pareschi, Lorenzo},
  journal={arXiv preprint arXiv:2506.10636},
  year={2025}
}

@article{jin2023asymptotic,
  title={Asymptotic-preserving neural networks for multiscale time-dependent linear transport equations},
  author={Jin, Shi and Ma, Zheng and Wu, Keke},
  journal={Journal of Scientific Computing},
  volume={94},
  number={3},
  pages={57},
  year={2023},
  publisher={Springer}
}

@article{jin2024asymptotic,
  title={Asymptotic-preserving neural networks for multiscale Vlasov--Poisson--Fokker--Planck system in the high-field regime},
  author={Jin, Shi and Ma, Zheng and Zhang, Tian-ai},
  journal={Journal of Scientific Computing},
  volume={99},
  number={3},
  pages={61},
  year={2024},
  publisher={Springer}
}

@article{wan2025error,
  title={Error estimates of asymptotic-preserving neural networks in approximating stochastic linearized Boltzmann equation},
  author={Wan, Jiayu and Liu, Liu},
  journal={arXiv preprint arXiv:2503.01643},
  year={2025}
}

@article{zhu2025physicssolver,
  title={PhysicsSolver: Transformer-Enhanced Physics-Informed Neural Networks for Forward and Forecasting Problems in Partial Differential Equations},
  author={Zhu, Zhenyi and Huang, Yuchen and Liu, Liu},
  journal={arXiv preprint arXiv:2502.19290},
  year={2025}
}

@article{lu2022neural,
  title={Neural Network Based Variational Methods for Solving Quadratic Porous Medium Equations in High Dimensions},
  author={Lu, Jianfeng and Wang, Min},
  journal={arXiv preprint arXiv:2205.02927},
  year={2022}
}

@article{abdo2024regularity,
  title={On the Regularity of Solutions to Some Active Scalar Equations on Bounded Domains},
  author={Abdo, Elie},
  journal={SIAM Journal on Mathematical Analysis},
  volume={56},
  number={6},
  pages={7306--7326},
  year={2024},
  publisher={SIAM}
}

@article{abdo2024dirichlet,
  title={On the Dirichlet Fractional Laplacian and Applications to the {SQG} Equation on Bounded Domains},
  author={Abdo, Elie and Lin, Quyuan},
  journal={arXiv preprint arXiv:2409.05209},
  year={2024}
}

@article{biswas2022error,
  title={Error estimates for deep learning methods in fluid dynamics},
  author={Biswas, Animikh and Tian, Jing and Ulusoy, Suleyman},
  journal={Numerische Mathematik},
  volume={151},
  number={3},
  pages={753--777},
  year={2022},
  publisher={Springer}
}

@article{benson2000application,
  title={Application of a fractional advection-dispersion equation},
  author={Benson, David A and Wheatcraft, Stephen W and Meerschaert, Mark M},
  journal={Water resources research},
  volume={36},
  number={6},
  pages={1403--1412},
  year={2000},
  publisher={Wiley Online Library}
}

@article{gilboa2009nonlocal,
  title={Nonlocal operators with applications to image processing},
  author={Gilboa, Guy and Osher, Stanley},
  journal={Multiscale Modeling \& Simulation},
  volume={7},
  number={3},
  pages={1005--1028},
  year={2009},
  publisher={SIAM}
}

@article{metzler2000random,
  title={The random walk's guide to anomalous diffusion: a fractional dynamics approach},
  author={Metzler, Ralf and Klafter, Joseph},
  journal={Physics reports},
  volume={339},
  number={1},
  pages={1--77},
  year={2000},
  publisher={Elsevier}
}

@article{cartea2007fractional,
  title={Fractional diffusion models of option prices in markets with jumps},
  author={Cartea, Alvaro and del-Castillo-Negrete, Diego},
  journal={Physica A: Statistical Mechanics and its Applications},
  volume={374},
  number={2},
  pages={749--763},
  year={2007},
  publisher={Elsevier}
}

@inproceedings{caffarelli2016fractional,
  title={Fractional elliptic equations, Caccioppoli estimates and regularity},
  author={Caffarelli, Luis A and Stinga, Pablo Ra{\'u}l},
  booktitle={Annales de l'Institut Henri Poincar{\'e} C, Analyse non lin{\'e}aire},
  volume={33},
  pages={767--807},
  year={2016},
  organization={Elsevier}
}

@article{constantin2017some,
  title={On some electroconvection models},
  author={Constantin, Peter and Elgindi, Tarek and Ignatova, Mihaela and Vicol, Vlad},
  journal={Journal of nonlinear science},
  volume={27},
  number={1},
  pages={197--211},
  year={2017},
  publisher={Springer}
}

@article{abdo2024error,
  title={Error estimates of physics-informed neural networks for approximating Boltzmann equation},
  author={Abdo, Elie and Chai, Lihui and Hu, Ruimeng and Yang, Xu},
  journal={To appear in IMA J. Numer. Anal.; arXiv:2407.08383},
  year={2024}
}

@article{raissi2019physics,
title={Physics-informed neural networks: A deep learning framework for solving forward and inverse problems involving nonlinear partial differential equations},
author={Raissi, Maziar and Perdikaris, Paris and Karniadakis, George E},
journal={Journal of Computational Physics},
volume={378},
pages={686--707},
year={2019}
}

@article{li2020fourier,
  title={Fourier neural operator for parametric partial differential equations},
  author={Li, Zongyi and Kovachki, Nikola and Azizzadenesheli, Kamyar and Liu, Burigede and Bhattacharya, Kaushik and Stuart, Andrew and Anandkumar, Anima},
  journal={arXiv preprint arXiv:2010.08895},
  year={2020}
}

@Article{E2018,
author={E, Weinan and Yu, Bing},
title={The Deep {Ritz} Method: A Deep Learning-Based Numerical Algorithm for Solving Variational Problems},
journal={Communications in Mathematics and Statistics},
year={2018},
volume={6},
number={1},
pages={1--12}
}

@Article{deepGalerkin2018,
author={Justin Sirignano and Konstantinos Spiliopoulos},
title={{DGM: A} deep learning algorithm for solving partial differential equations},
journal={Journal of Computational Physics},
volume={375},
year={2018},
pages={1339--1364}
}

@article{zang2020weak,
title={Weak adversarial networks for high-dimensional partial differential equations},
author={Zang, Yaohua and Bao, Gang and Ye, Xiaojing and Zhou, Haomin},
journal={Journal of Computational Physics},
pages={109409},
year={2020},
publisher={Elsevier}
}

@article{cai2021least,
  title={Least-squares {ReLU neural network (LSNN)} method for linear advection-reaction equation},
  author={Cai, Zhiqiang and Chen, Jingshuang and Liu, Min},
  journal={Journal of Computational Physics},
  pages={110514},
  year={2021}
}

@article{lu2021learning,
  title={Learning nonlinear operators via DeepONet based on the universal approximation theorem of operators},
  author={Lu, Lu and Jin, Pengzhan and Pang, Guofei and Zhang, Zhongqiang and Karniadakis, George Em},
  journal={Nature Machine Intelligence},
  volume={3},
  pages={218--229},
  year={2021}
}

@article{beck2020overview,
  title={An overview on deep learning-based approximation methods for partial differential equations},
  author={Beck, Christian and Hutzenthaler, Martin and Jentzen, Arnulf and Kuckuck, Benno},
  journal={arXiv preprint arXiv:2012.12348},
  year={2020}
}

@article{hao2024structure,
  title={Structure preserving PINN for solving time dependent PDEs with periodic boundary},
  author={Hao, Baoli and Braga-Neto, Ulisses and Liu, Chun and Wang, Lifan and Zhong, Ming},
  journal={arXiv preprint arXiv:2404.16189},
  year={2024}
}

@article{abdo2025neural,
  title={Neural network solutions to the critical SQG equations via approximating nonlocal periodic operators},
  author={Abdo, Elie and Hu, Ruimeng and Lin, Quyuan},
  journal={Physica D: Nonlinear Phenomena},
  volume={476},
  pages={134652},
  year={2025},
  publisher={Elsevier}
}

@article{hu2023higher,
  title={Higher-order error estimates for physics-informed neural networks approximating the primitive equations},
  author={Hu, Ruimeng and Lin, Quyuan and Raydan, Alan and Tang, Sui},
  journal={Partial Differential Equations and Applications},
  volume={4},
  number={4},
  pages={34},
  year={2023},
  publisher={Springer}
}

\end{document}